\documentclass[12pt]{article}

\usepackage{authblk}

\usepackage{amssymb,amsmath,amsthm}
\usepackage{amsfonts}
\usepackage{etoolbox}
\usepackage{bbm}
\usepackage{multirow}
\usepackage{enumerate}
\usepackage{url}
\usepackage{mathtools}
\usepackage{array}

\usepackage{palatino}
\usepackage{fancyhdr}	
\usepackage{mdframed}
\usepackage{geometry}
\newcommand{\Mod}[1]{\ (\mathrm{mod}\ #1)}

\newtheoremstyle{mytheoremstyle} 
    {10pt}                    
    {10pt}                    
    {\normalfont}                   
    {}                           
    {\bfseries}                   
    {.}                          
    {0.3cm}                       
    {}  
\theoremstyle{mytheoremstyle}

\theoremstyle{plain}

\fancypagestyle{plain}{%
  \fancyhf{}%
}

\newtheorem{theorem}{Theorem}[section]

\newtheorem{lemma}[theorem]{Lemma}
\newtheorem{proposition}[theorem]{Proposition}

\newtheorem{example}[theorem]{Example}

\begin{document}
\title{Quasi-orthogonal extension of symmetric matrices}
	\author[,1]{Abderrahim Boussa\"{\i}ri\thanks{Corresponding author, email: aboussairi@hotmail.com}}  
	\author[1]{Brahim Chergui}
	\author[1]{Zaineb Sarir}
	\author[2]{Mohamed Zouagui}

\affil[1]{Laboratoire Math\'ematiques Fondamentales et Appliqu\'ees, Facult\'e des Sciences A\"in Chock, Hassan II University of Casablanca, Morocco}

\affil[2]{Laboratoire de Recherche Math\'ematiques et Sciences de l'Ing\'enieur MSI, International University of Casablanca,
	School of Engineering, 
	Casa Green Town,
	Bouskoura, Morocco}

\maketitle

\begin{abstract}
	An $n\times n$ real matrix $Q$ is quasi-orthogonal if $Q^{\top}Q=qI_{n}$ for some positive real number $q$. If $M$ is a principal sub-matrix of a quasi-orthogonal matrix $Q$, we say that $Q$ is a quasi-orthogonal extension of $M$. In a recent work, the authors have investigated this notion for the class of real skew-symmetric matrices. Using a different approach, this paper addresses the case of symmetric matrices. 	 
\end{abstract}



\section{Introduction}

All the matrices considered in this paper are real. The identity matrix of order $n$ and the $n\times m$ all-zeros matrix are respectively denoted by $I_{n}$ and $O_{n,m}$. We omit the subscript when the order is understood. The characteristic polynomial of an $n\times n$ matrix $A$ is $\phi_{A}(x):=\det(xI_{n}-A)$. 

An $n\times n$ real matrix $Q$ is \emph{quasi-orthogonal} if $Q^{\top}Q=qI_{n}$, for some positive real number $q$ or equivalently, the matrix $\frac{1}{\sqrt{q}}Q$ is orthogonal. Special classes of quasi-orthogonal matrices are the set of Hadamard matrices and conference matrices. Recall that a \emph{Hadamard matrix} is a square matrix with entries in $\{-1,1\}$ and whose columns are mutually orthogonal. The order of such matrices must be 1, 2 or a multiple of 4. A \emph{conference matrix} is an $n\times  n$ matrix $C$ with 0 on the diagonal and $\pm 1$ off the diagonal such that $C^{\top}C=(n-1)I_{n}$. Hadamard and conference matrices have been extensively studied since they are related to many combinatorial problems. We refer to the Handbook of Combinatorial Designs \cite{ch2007handbook} for relevant background. 
Taussky \cite{taussky1971sums} suggested the following generalization of Hadamard and conference matrices.
A \emph{weighing matrix} of weight $k$ and order $n$ is an $n \times n$ $\{-1,0,1\}$ matrix $A$ such that $AA^{\top} = kI_{n}$. Recent results on these topics can be found in \cite{koukouvinos1999new}.

Let $M$ be a square matrix. If $M$ is a principal sub-matrix of a quasi-orthogonal matrix $Q$, we say that $Q$ is a \emph{quasi-orthogonal extension} of $M$. Any symmetric (resp. skew-symmetric) matrix with $0$ in the diagonal and $\pm 1$ off the diagonal is a principal sub-matrix of a symmetric (resp. skew-symmetric) conference matrix. These facts follow from \cite[Theorem 3]{bollobas1981graphs} and \cite{graham1971constructive}. 	In \cite{boussairi2024quasi}, the authors proved that every skew-symmetric matrix has a skew-symmetric quasi-orthogonal extension.

This paper deals with the real symmetric matrices. Let $M$ be a symmetric matrix. We prove in Section \ref{sectionExistance} that  $M$ has a symmetric quasi-orthogonal extension. 
We define the \emph{quasi-orthogonality index} of $M$ as the least integer $d$, denoted by ${\rm ind}(M)$, such that $M$ has a symmetric quasi-orthogonal extension of order $n+d$.

The remainder of the paper is organized as follows.
In Section \ref{sectionIndex}, we work out the quasi-orthogonality index using the well-known Cauchy's interlace theorem. Section \ref{sectionSeidel} is devoted to $n\times n$ Seidel matrices with minimum quasi-orthogonality index.

\section{Existence of symmetric quasi-orthogonal extension of symmetric matrices}\label{sectionExistance}
Let $S$ be a non-zero square symmetric matrix. For a real number $\lambda$, we denote by $\mu_{\lambda}(S)$ (or simply by $\mu_{\lambda}$) its algebraic multiplicity in $S$ if $\lambda$ is an eigenvalue of $S$. We convene that $\mu_{\lambda}=0$ whenever $\lambda$ is not an eigenvalue of $S$.

The next theorem guarantees the existence of symmetric quasi-orthogonal extension of symmetric matrices. 
\begin{theorem}\label{existance}
	Let $S$ be an $n\times n$ non-zero real symmetric matrix with spectral radius $\rho$. Then $S$ has a symmetric quasi-orthogonal extension $\hat{S}$ of order $2n-\mu_{\rho}-\mu_{-\rho}$.
\end{theorem}

\begin{proof}
	Since the matrix $S$ is real and symmetric, it is diagonalizable, and hence there exists an $n\times n$ orthogonal matrix $P$ such that
	\begin{equation}\label{decomposition}
		P^{\top}SP=\begin{pmatrix}
			D    & O\\
			O  & \Sigma
		\end{pmatrix}
	\end{equation} 
	where $D$ is a diagonal matrix and $\Sigma$ is a diagonal matrix of order $\mu_{\rho}+\mu_{-\rho}$ with diagonal entries $\pm\rho$. 
	Then we have     $$ P^{\top}S^{2}P=\begin{pmatrix}
		D^{2}    & O\\
		O  & \rho^{2}I
	\end{pmatrix}.$$
	We partition the matrix $P$ as $P:=\begin{pmatrix}
		N & L
	\end{pmatrix}$, where $N$ and $L$ are sub-matrices of orders $n\times(n-\mu_{-\rho}-\mu_{\rho})$ and $n\times(\mu_{-\rho}+\mu_{\rho})$ respectively.
	Note that the diagonal matrix $\rho^{2}I-D^{2}$ has positive diagonal entries.
	We set $M:=N.\sqrt{\rho^{2}I-D^{2}}$.
	Let us consider the following extension $\hat{S}$ of $S$
	$$\hat{S}:=\begin{pmatrix}
		S  & M\\
		M^{\top}  & -D
	\end{pmatrix}.$$
	We have
	$$\hat{S}^{2}=\begin{pmatrix}
		S^{2}+MM^{\top}  & SM-MD\\
		M^{\top}S-DM^{\top}  & M^{\top}M+D^{2}
	\end{pmatrix}.$$
	Equality \eqref{decomposition} implies that $SN=ND$. Right multiplying both sides by $\sqrt{\rho^{2}I-D^{2}}$, we get $SM=MD$. Hence
	\begin{equation}\label{eq1}
		SM-MD=M^{\top}S-DM^{\top}=O.
	\end{equation}
	As $P$ is an orthogonal matrix and $P^{\top}P=\begin{pmatrix}
		N^{\top}N    & N^{\top}L\\
		L^{\top}N  & L^{\top}L
	\end{pmatrix}$, we obtain $N^{\top}N=I$.
	It follows that
	$M^{\top}M=\rho^{2}I-D^{2}.$
	Then 
	\begin{equation}\label{eq2}
		M^{\top}M+D^{2}=\rho^{2}I.
	\end{equation}
	Moreover
	\begin{align*}\label{eq1}
		MM^{\top}&=N(\rho^{2}I-D^{2})N^{\top}\\
		&=\rho^{2}NN^{\top}-ND^{2}N^{\top}\\
	\end{align*}
	and 
	\begin{align*}
		S^{2}&=P\begin{pmatrix}
			D^{2}    & O\\
			O  & \rho^{2}I
		\end{pmatrix}P^{\top}\\
		&=\begin{pmatrix}
			N & L
		\end{pmatrix}\begin{pmatrix}
			D^{2}    & O\\
			O  & \rho^{2}I
		\end{pmatrix}\begin{pmatrix}
			N^\top\\
			L^\top
		\end{pmatrix}\\
		&=ND^{2}N^{\top}+\rho^{2}LL^{\top}.
	\end{align*}
	It follows that
	$S^{2}+ MM^{\top} =\rho^{2}(NN^{\top}+LL^{\top})$.
	Given that $PP^{\top}=NN^{\top}+LL^{\top}=I$, we have 
	\begin{equation}\label{eq3}
		S^{2}+ M^{\top}M=\rho^2I.	
	\end{equation}
	From \eqref{eq1}, \eqref{eq2} and \eqref{eq3} we deduce that $\hat{S}$ is a symmetric quasi-orthogonal extension of order $2n-\mu_{-\rho}-\mu_{\rho}$.
\end{proof}
\section{Quasi-orthogonality index of a symmetric matrix}\label{sectionIndex}
The following theorem gives the quasi-orthogonality index for real symmetric matrices.
\begin{theorem}\label{index}
	Let $S$ be an $n\times n$ symmetric matrix with spectral radius $\rho$.
	The quasi-orthogonality index of $S$ is equal to $n-\mu_{\rho}-\mu_{-\rho}$.
\end{theorem}
To prove this theorem, we need the following results.
\begin{theorem}[Cauchy's interlace theorem]\label{interlacing}
	Let $A$ be an $n\times n$ symmetric matrix and let $B$ be a principal sub-matrix of order $m<n$. Suppose $A$ has eigenvalues $\lambda_{1}\leq\lambda_{2}\leq\cdots\leq\lambda_{n}$ and $B$ has eigenvalues $\beta_{1}\leq\beta_{2}\leq\cdots\leq\beta_{m}$. Then
	$$\lambda_{k}\leq\beta_{k}\leq\lambda_{k+n-m} \ \ \ \text{for} \ \ k=1,\dots,m,$$
	in particular, if $m=n-1$, we have
	\[\lambda_{1}\leq\beta_{1}\leq\lambda_{2}\leq\beta_{2}\leq\dots\leq\beta_{n-1}\leq\lambda_{n}.
	\]	
\end{theorem}

\begin{lemma}\label{f(s)}
	Let $S$ be a real symmetric matrix of order $m$ with spectral radius $\rho$.
	Let $\hat{S}$ be a symmetric quasi-orthogonal extension of order $n\geq 2$, with spectral radius $\hat{\rho}$. 
	If $n-m\leq (n-1)/2$, then $\hat{\rho}=\rho$ and $n\geq  2m-(\mu_{\rho}(S)+\mu_{-\rho}(S))$. 
\end{lemma}
\begin{proof}
	Let $\lambda_{1}\leq\lambda_{2}\leq\cdots\leq\lambda_{n}$ and $\beta_{1}\leq\beta_{2}\leq\cdots\leq\beta_{m}$ respectively be the eigenvalues of $\hat{S}$ and $S$.	
	We set $f(S):=m-(\mu_{\rho}(S)+\mu_{-\rho}(S))$.
	Without loss of generality we can assume that $\mu_{-\hat{\rho}}(\hat{S})\geq \mu_{\hat{\rho}}(\hat{S})$. We will distinguish two cases:
	
	$\bullet$ Suppose that $n-m\geq\mu_{\hat{\rho}}(\hat{S})$. As $\hat{S}$ is a symmetric quasi-orthogonal matrix, $\mu_{\hat{\rho}}(\hat{S})+\mu_{-\hat{\rho}}(\hat{S})=n$. It follows that
	$$2\mu_{-\hat{\rho}}(\hat{S})\geq n.$$
	Hence
	$$\mu_{-\hat{\rho}}(\hat{S})>\dfrac{n-1}{2}\geq n-m.$$
	Let $k\in\{1,\dots,\mu_{-\hat{\rho}}(\hat{S})-(n-m)\}$. As $k+n-m\leq \mu_{-\hat{\rho}}(\hat{S})$, we have $\lambda_{k}=\lambda_{k+n-m}=-\hat{\rho}$.
	By Theorem \ref{interlacing}, we get $\beta_{k}=-\hat{\rho}$. Then $\hat{\rho}=\rho$. Moreover, we have
	$$f(S)\leq m-(\mu_{-\hat{\rho}}(\hat{S})-(n-m)).$$
	Then
	$$f(S)\leq n-\mu_{-\hat{\rho}}(\hat{S}).$$
	Hence
	$$f(S)\leq \mu_{\hat{\rho}}(\hat{S})\leq n-m.$$
	Consequently, $n\geq  2m-(\mu_{\rho}(S)+\mu_{-\rho}(S))$.
	
	$\bullet$ Suppose that $n-m<\mu_{\hat{\rho}}(\hat{S})$, then $n-m<\mu_{-\hat{\rho}}(\hat{S})$. 
	
	If $k\in\{1,\dots,\mu_{-\hat{\rho}}(\hat{S})-(n-m)\}$, then we have $k+n-m\leq \mu_{-\hat{\rho}}(\hat{S})$ and $\lambda_{k}=\lambda_{k+n-m}=-\hat{\rho}$. It follows from Theorem  \ref{interlacing} that $\beta_{k}=-\hat{\rho}$ for $k\in\{1,\dots,\mu_{-\hat{\rho}}(\hat{S})-(n-m)\}$. 
	
	If $k\in\{\mu_{-\hat{\rho}}(\hat{S})+1,\dots,m\}$, then $\lambda_{k}=\lambda_{k+n-m}=\hat{\rho}$ and hence $\beta_{k}=\hat{\rho}$. Therefore, $f(S)\leq m-[\mu_{-\hat{\rho}}(\hat{S})-(n-m)+m-\mu_{-\hat{\rho}}(\hat{S})]=n-m$. Consequently, $\hat{\rho}=\rho$ and $n\geq  2m-(\mu_{\rho}(S)+\mu_{-\rho}(S))$.
\end{proof}

\begin{proof}[Proof of Theorem \ref{index}]
	Let $\hat{S}$ be  a minimal symmetric quasi-orthogonal extension of the matrix $S$ with order $\hat{n}$. 
	By Theorem \ref{existance}, $S$ has a symmetric quasi-orthogonal extension of order $2n-\mu_{\rho}(S)-\mu_{-\rho}(S)$. Then
	$$\hat{n}\leq2n-\mu_{\rho}(S)-\mu_{-\rho}(S).$$
	It follows that
	$$\hat{n}-2n\leq-\mu_{\rho}(S)-\mu_{-\rho}(S)\leq -1,$$
	and hence
	$$\hat{n}-n\leq\dfrac{\hat{n}-1}{2}.$$
	By Lemma \ref{f(s)}, we have $\hat{n}\geq2n-\mu_{\rho}(S)-\mu_{-\rho}(S)$. Thus $\hat{n}=2n-\mu_{\rho}(S)-\mu_{-\rho}(S)$.	
\end{proof}

\begin{example} \label{exampleordr3}
	We will use the proof of Theorem \ref{existance} to find a minimal symmetric quasi-orthogonal extension of the matrix
	$$S=
	\begin{pmatrix}
		0 & 1 & 1\\
		1 & 0 & 1\\
		1 & 1 & 0
	\end{pmatrix}.$$
	The spectrum of $S$ is $\{[-1]^{2}, [2]^{1}\}$. By Theorem \ref{index}, we have ${\rm ind}(S)=2$.
	Let $$P=\begin{pmatrix}
		\frac{-1}{\sqrt{2}} & \frac{1}{\sqrt{6}} &	\frac{1}{\sqrt{3}}\\
		\frac{1}{\sqrt{2}} & \frac{1}{\sqrt{6}}& 	\frac{1}{\sqrt{3}}\\
		0 & \frac{-2}{\sqrt{6}} & 	\frac{1}{\sqrt{3}}
	\end{pmatrix}.$$ 
	The matrix $P$ is orthogonal and
	$$P^\top SP=
	\begin{pmatrix*}[r]
		-1 & 0 & 0\\
		0 & -1 & 0\\
		0 & 0 & 2
	\end{pmatrix*}.$$
	Following the proof of Theorem \ref{existance}, we have 
	$N=\begin{pmatrix}
		\frac{-1}{\sqrt{2}} & \frac{1}{\sqrt{6}}\\
		\frac{1}{\sqrt{2}} & \frac{1}{\sqrt{6}}\\
		0 & \frac{-2}{\sqrt{6}}%
	\end{pmatrix}$, 
	$L=\begin{pmatrix}
		\frac{1}{\sqrt{3}}\\
		\frac{1}{\sqrt{3}}\\
		\frac{1}{\sqrt{3}}%
	\end{pmatrix}$,
	$D=\begin{pmatrix*}[r]
		-1 & 0\\
		0 & -1
	\end{pmatrix*}$
	and $M=N.\sqrt{4I-D^{2}}$.
	A minimal symmetric quasi-orthogonal extension of $S$ is then 
	$$\hat{S}=
	\begin{pmatrix*}[r]
		0 & 1 & 1 & -\frac{1}{2}\sqrt{6} & \frac{1}{2}\sqrt{2}\\
		1 & 0 & 1 & \frac{1}{2}\sqrt{6} & \frac{1}{2}\sqrt{2}\\
		1 & 1 & 0 & 0 & -\sqrt{2} \\
		-\frac{1}{2}\sqrt{6} & \frac{1}{2}\sqrt{6} & 0 & 1 & 0\\
		\frac{1}{2}\sqrt{2} & \frac{1}{2}\sqrt{2} & -\sqrt{2} & 0 & 1
	\end{pmatrix*}.$$
\end{example}

\section{Symmetric Seidel matrices with minimum quasi-orthogonality index}\label{sectionSeidel}
Following \cite{greaves2017symmetric}, a Seidel matrix is a $\{0,\pm1\}$-matrix $S$ with zero diagonal and all off-diagonal entries nonzero such that $S = \pm S^{\top}$. Symmetric Seidel matrices were introduced in \cite{van1991equilateral} in connection with equiangular lines in Euclidean spaces \cite{haantjes1948equilateral}.

Let $S$ be a symmetric Seidel matrix of order $n$ with spectrum $\{[\theta_{1}]^{m_{1}},\dots, [\theta_{r}]^{m_{r}}\}$. Then
\begin{align}
	&\sum_{i=1}^{r}m_{i}=n,\\
	&\sum_{i=1}^{r}m_{i}\theta_{i}=0,\label{trace}\\
	&\sum_{i=1}^{r}m_{i}\theta_{i}^{2}=n(n-1).\label{symmetricMatrix}
\end{align}
Two Seidel matrices $S_1$ and $S_2$ are said to be \emph{equivalent} if there exists a signed permutation matrix $P$ such that $S_2 = \pm PS_1P^{\top}$. Two equivalent Seidel matrices have the same quasi-orthogonality index.

For an integer $n\geq 2$, we denote by ${ \rm ind}(n)$ the minimum quasi-orthogonality index among all $n\times n$ Seidel matrices. 

The two Seidel matrices of order 2 are orthogonal and then ${ \rm ind}(2)=0$. Up to equivalence, there is only one Seidel matrix of order 3, namely 
$$S=
\begin{pmatrix}
	0 & 1 & 1\\
	1 & 0 & 1\\
	1 & 1 & 0
\end{pmatrix}.$$ As seen in Example \ref{exampleordr3}, ${ \rm ind}(S)=2$, hence ${ \rm ind}(3)=2$.
Table \ref{tab:matrices_polynomials} lists all non equivalent $n\times n$ Seidel matrices having the minimum quasi-orthogonality index for $n\in\{4, 5, 6, 7\}$.

\begin{table}[!h]
	\caption{Non equivalent $n\times n$ Seidel matrices with minimum quasi-orthogonality index for $n\in\{4, 5, 6, 7\}$.}
	\begin{small}
		\begin{tabular}{|c|c|c|c|}
			\hline
			\textbf{Order $n$} & \textbf{${\rm ind}(n)$}&\textbf{Seidel matrix} & \textbf{Characteristic polynomial} \\
			\hline
			$4$& 2&
			$\begin{pmatrix*}[r]
				0 & 1 & -1 & 1 \\
				1 & 0 & 1 & 1 \\
				-1 & 1 & 0 & 1 \\
				1 & 1 & 1 & 0
			\end{pmatrix*}$
			&
			$(x^{2} - 1)(x^2 - 5)$\\
			\hline
			$5$& 1&
			$\begin{pmatrix*}[r]
				0 & 1 & -1 & -1 & 1 \\
				1 & 0 & 1 & -1 & 1 \\
				-1 & 1 & 0 & 1 & 1 \\
				-1 & -1 & 1 & 0 & 1 \\
				1 & 1 & 1 & 1 & 0
			\end{pmatrix*}$
			&
			$x(x^2 - 5)^2$\\
			\hline
			$6$ & 0
			&
			$\begin{pmatrix*}[r]
				0 & 1 & -1 & -1 & 1 & 1 \\
				1 & 0 & 1 & -1 & -1 & 1 \\
				-1 & 1 & 0 & 1 & -1 & 1 \\
				-1 & -1 & 1 & 0 & 1 & 1 \\
				1 & -1 & -1 & 1 & 0 & 1 \\
				1 & 1 & 1 & 1 & 1 & 0
			\end{pmatrix*}
			$
			&
			$(x^2 - 5)^3$
			\\
			\hline
			$7$ & 1 
			& 
			$
			\begin{pmatrix*}[r]
				0 & 1 & 1 & -1 & -1 & 1 & 1 \\
				1 & 0 & 1 & -1 & 1 & -1 & 1 \\
				1 & 1 & 0 & 1 & -1 & -1 & 1 \\
				-1 & -1 & 1 & 0 & 1 & 1 & 1 \\
				-1 & 1 & -1 & 1 & 0 & 1 & 1 \\
				1 & -1 & -1 & 1 & 1 & 0 & 1 \\
				1 & 1 & 1 & 1 & 1 & 1 & 0
			\end{pmatrix*}
			$ 
			& 
			$(x + 2)(x - 1)^2(x^{2} -9)^2$\\
			\hline
		\end{tabular}
	\end{small}
	\label{tab:matrices_polynomials}
\end{table}

Symmetric quasi-orthogonal Seidel matrices are known as symmetric conference matrices. The order $n$ of a symmetric conference matrix is of the form $4k+2$ for some integer $k$, moreover, $n - 1$ is the sum of two squares \cite{belevitch1950theorem}.
The first orders of symmetric conference matrices are $n = 2, 6, 10, 14, 18$. However, there is no conference matrix of order 22 (21 is not the sum of two squares).
Greaves and Suda \cite{greaves2017symmetric} prove that the existence of a symmetric conference matrix of order $4k + 2$ is equivalent to the existence of a symmetric Seidel matrix of order $4k + i$ having a prescribed spectrum for each $i \in \{1, 0, -1\}$. More precisely, they obtain the following result.
\begin{theorem}[\cite{greaves2017symmetric}]\label{suda}
	The existence of the following are equivalent: 
	\begin{enumerate}[i)]
		\item a symmetric Seidel matrix with characteristic polynomial $$(x^{2} - 4k - 1)^{2k+1};$$
		\item a symmetric Seidel matrix with characteristic polynomial $$x(x^{2} - 4k - 1)^{2k};$$
		\item a symmetric Seidel matrix with characteristic polynomial $$(x^{2} - 1)(x^{2} - 4k - 1)^{2k-1};$$
		\item  a symmetric Seidel matrix with characteristic polynomial $$(x-2)(x + 1)^{2}(x^{2} - 4k -1)^{2(k -1)}.$$
	\end{enumerate}
\end{theorem}
It follows from Theorem \ref{index} that the quasi-orthogonality index of the symmetric Seidel matrices mentioned in assertions $i)$, $ii)$, $iii)$ and $iv)$ of the above theorem are respectively equal to 0, 1, 2 and 3.		

Seidel matrices with quasi-orthogonality index 0 are conference matrices.
The following result gives the characteristic polynomials of symmetric Seidel matrices with quasi-orthogonality index 1.

\begin{proposition}\label{Seidel index 1}
	Let $S$ be an $n\times n$ Seidel matrix with quasi-orthogonality index equals $1$. Then $n$ is odd and its characteristic polynomial  $\phi_{S}=x(x^{2}-n)^{\frac{n-1}{2}}$.
\end{proposition}
\begin{proof}
	Let $\rho$ be the spectral radius of $S$. As ${\rm ind}(S)=1$ then $S$ has a simple eigenvalue $\alpha$ such that $\lvert \alpha \rvert<\rho$. 
	If $\mu_{-\rho}=0$ then by \eqref{trace} we have $(n-1)\rho + \alpha = 0$, which is a contradiction because $\lvert \alpha \rvert<\rho$. We obtain a similar contradiction if $\mu_{\rho}=0$.
	Now suppose that $\mu_{-\rho}\neq\mu_{\rho}$. By \eqref{trace} we get $(\mu_{\rho}-\mu_{-\rho})\rho + \alpha=0$, 
	which implies that $\lvert \mu_{\rho}-\mu_{-\rho}\rvert\rho=\lvert \alpha\rvert$. This is a contradiction. Then $\mu_{-\rho}=\mu_{\rho}$ and hence $\alpha=0$ and $n=2\mu_{\rho}+1$. By \eqref{symmetricMatrix} we have $\rho^{2}=n$.
	Then the characteristic polynomial $\phi_{S}(x)=x(x^{2}-n)^{\frac{n-1}{2}}$.
\end{proof}

For a given $n\geq4$, the following theorem provides a spectral characterization of $n\times n$ Seidel matrices with minimum quasi-orthogonality index.
\begin{theorem}\label{Seidel_thm}
	Let $S$ be a Seidel matrix of order $n\geq 4$. Then, we have the following statements
	\begin{enumerate}[i)]
		\item If $n=4k+2$, then ${\rm ind}(S)\geq0$. Equality holds if and only if $\phi_{S}(x)=(x^{2} - 4k - 1)^{2k+1}$.
		\item If $n=4k+1$, then ${\rm ind}(S)\geq1$. Equality holds if and only if $\phi_{S}(x)=x(x^{2} - 4k - 1)^{2k}.$
		\item If $n=4k$, then ${\rm ind}(S)\geq2$. Equality holds if and only if $\phi_{S}(x)=(x^{2} - 1)(x^{2} - 4k - 1)^{2k-1}$.
		\item If $n=4k-1$, then ${\rm ind}(S)\geq3$. Equality holds if and only if $\phi_{S}(x)=(x-2)(x + 1)^{2}(x^{2} - 4k -1)^{2(k -1)}$ or $\phi_{S}(x)=(x+2)(x - 1)^{2}(x^{2} - 4k -1)^{2(k -1)}$.
	\end{enumerate}
\end{theorem}
To prove this theorem, we need the following results.
\begin{lemma}\label{integercoefficient}
	Let $P(x)=(x-a_{1})^{\alpha_{1}}(x-a_{2})^{\alpha_{2}}\dots(x-a_{r})^{\alpha_{r}}$ be an integral polynomial where $a_{1},\dots, a_{r}$ are distinct real numbers and $\alpha_{1}\geq\alpha_{2}\geq\dots\geq\alpha_{r}$. If $\alpha_{1}=\alpha_{2}=\dots=\alpha_{k}>\alpha_{k+1}$ then $(x-a_{1})(x-a_{2})\dots(x-a_{k})$ is an integral polynomial.
	
\end{lemma}
\begin{proof}
	The result follows from the fact that if $a_{i}$ and $a_{j}$ share the same minimal polynomial then $\alpha_{i}=\alpha_{j}$.
\end{proof}
The next proposition gives a useful property about the determinant of symmetric Seidel matrices.
\begin{proposition}\cite[Corollary 3.6]{greaves2016equiangular}\label{det}
	Let $S$ be a symmetric Seidel matrix of order $n$. Then  $\det(S) \equiv 1 - n \pmod 4$.
\end{proposition}

\begin{proof}[Proof of Theorem \ref{Seidel_thm}]
	For $n\in\{4, 5, 6, 7\}$, the proof follows from Table $\ref{tab:matrices_polynomials}$. Now, we assume that $n\geq8$.
	\begin{enumerate}[i)]
		\item If ${\rm ind}(S)=0$, then $S$ is quasi-orthogonal and hence $\phi_{S}(x)=(x^{2}-4k-1)^{2k+1}$.
		\item If ${\rm ind}(S)=0$, then $S$ is quasi-orthogonal which implies that $n$ is even, a contradiction. Hence ${\rm ind}(S)\geq1$.	
		If ${\rm ind}(S)=1$, then by Proposition \ref{Seidel index 1}, we have  $\phi_{S}(x)=x(x^{2}-4k-1)^{2k}$. 
		
		\item If ${\rm ind}(S)=0$, then $S$ is quasi-orthogonal. Hence $S$ is a symmetric conference matrix. This is a contradiction because $n\not\equiv2\Mod{4}$. Consequently ${\rm ind}(S)\geq1$.
		
		If ${\rm ind}(S)=1$, then by Proposition \ref{Seidel index 1}, $n$ is odd which is a contradiction. Then ${\rm ind}(S)>1$.

		If ${\rm ind}(S)=2$ then $\phi_{S}(x)=(x-\alpha)(x-\beta)(x-\rho)^{\mu_{\rho}}(x+\rho)^{\mu_{-\rho}}$ where $\lvert \alpha \rvert<\rho$ and $\lvert \beta \rvert<\rho$. By \eqref{trace} we have $\alpha+\beta+(\mu_{\rho}-\mu_{-\rho})\rho=0$. It follows that $\lvert \mu_{\rho}-\mu_{-\rho}\rvert<2$. Moreover, $\mu_{\rho}+\mu_{-\rho}=4k-2$ then $\mu_{-\rho}$ and $\mu_{\rho}$ have the same parity. Hence $\mu_{\rho}=\mu_{-\rho}$ and $\beta=-\alpha$. 
		By \eqref{symmetricMatrix}, we get
		\begin{equation}
			\alpha^{2}+(2k-1)\rho^{2}=2k(4k-1)\label{eq6}
		\end{equation}
		Then
		\begin{align*}
			(2k-1)\rho^{2}\leq2k(4k-1).
		\end{align*}
		Hence
		\begin{equation*}
			\rho^{2}\leq 4k+1+ \dfrac{1}{2k-1}.
		\end{equation*}
		By Lemma \ref{integercoefficient}, we have $\alpha^{2}$ and $\rho^{2}$ are integers. As $n \geq 8$, we get $k\geq 2$ and hence 
		\begin{equation}
			\rho^{2}\leq4k+1.
		\end{equation}
		Since $\lvert \alpha \rvert<\rho$, equality \eqref{eq6} implies that
		$$2k\rho^{2}>2k(4k-1).$$
		Then
		$$\rho^{2}>4k-1.$$
		We cannot have $\rho^{2}=4k$, otherwise $\alpha^{2}=2k$ and then $\det(S)=2k(4k)^{2k-1}$ which contradicts Proposition \ref{det}. Therefore, $\rho^{2}=4k+1$ and consequently $\alpha^{2}=1$.
		Hence, $\phi_{S}(x)=(x^{2} - 1)(x^{2} - 4k - 1)^{2k-1}$.
		\item As $n$ is odd, $S$ cannot be quasi-orthogonal. Therefore, ${\rm ind}(S)\geq1$.
		
		If ${\rm ind}(S)=1$, then by Proposition \ref{Seidel index 1}, ${\rm det}(S)=0$. However, Proposition \ref{det} implies that ${\rm det}(S)\equiv 2 \pmod 4$. This is a contradiction. Consequently, ${\rm ind}(S)\geq2$.
		
		If ${\rm ind}(S)=2$ then $\phi_{S}(x)=(x-\alpha)(x-\beta)(x-\rho)^{\mu_{\rho}}(x+\rho)^{\mu_{-\rho}}$ where $\lvert \alpha \rvert<\rho$ and $\lvert \beta \rvert<\rho$. By \eqref{symmetricMatrix},  we have
		\begin{equation}\label{eq10}
			\alpha^{2}+\beta^{2}+(4k-3)\rho^{2}=(4k-1)(4k-2).
		\end{equation}
		Then
		$$(4k-1)\rho^{2}>(4k-1)(4k-2).$$
		Hence
		$$\rho^{2}>4k-2.$$
		Moreover
		$$(4k-3)\rho^{2}\leq(4k-1)(4k-2).$$
		Thus $$\rho^{2}\leq 4k+\dfrac{2}{4k-3}.$$
		Moreover, $\mu_{-\rho}$ and $\mu_{\rho}$ have different parity because $\mu_{\rho}+\mu_{-\rho}=4k-3$. It follows from Lemma \ref{integercoefficient} that $\rho$ is an integer. Since $n\geq 8$, we have $k\geq 3$ and hence
		\begin{equation}\label{eq11}
			4k-1\leq\rho^{2}\leq4k.
		\end{equation}		 
		By \eqref{trace} we have $\alpha+\beta+(\mu_{\rho}-\mu_{-\rho})\rho=0$. Thus $\lvert \mu_{\rho}-\mu_{-\rho}\rvert<2$. Hence $\lvert \mu_{\rho}-\mu_{-\rho}\rvert=1$ and $\alpha+\beta=\pm\rho$. Then $(\alpha+\beta)^{2}=\rho^{2}$. Using \eqref{eq10}, we have $\alpha\beta=(2k-1)(\rho^{2}-4k+1)$. Consequently $\alpha\beta$ is an integer. 
		From Proposition \ref{det}, we have $\det(S)=\pm\alpha\beta\rho^{4k-3}\equiv 2\pmod4$. Then $\rho$ must be odd, otherwise $\rho^{2}\equiv 0\pmod4$ and ${\rm det}(S)\equiv0\pmod4$. By \eqref{eq11} we get $\rho^{2}=4k-1$. Then $\alpha\beta=0$ and ${\rm det}(S)=0$, a contradiction. Hence ${\rm ind}(S)\geq 3$.
		
		If ${\rm ind}(S)=3$, then $\phi_{S}(x)=(x-\alpha)(x-\beta)(x-\gamma)(x-\rho)^{\mu_{\rho}}(x+\rho)^{\mu_{-\rho}}$ where $\lvert \alpha \rvert<\rho$, $\lvert \beta \rvert<\rho$ and $\lvert \gamma \rvert<\rho$. By \eqref{symmetricMatrix}, we get 
		\begin{equation}\label{eq12}
			\alpha^{2}+\beta^{2}+\gamma^{2}+(4k-4)\rho^{2}=(4k-1)(4k-2).
		\end{equation}
		This implies 
		$$ 4k-2<\rho^{2}\leq\frac{(4k-1)(4k-2)}{4k-4}.$$
		Then
		\begin{equation}\label{eq15}
			4k-1\leq\rho^{2}\leq4k+1 +\frac{6}{4k-4}.
		\end{equation}
		
		By \eqref{trace} we have $\alpha+\beta+\gamma+(\mu_{\rho}-\mu_{-\rho})\rho=0$. It follows that $\lvert \mu_{\rho}-\mu_{-\rho}\rvert<3$. In addition, $\mu_{\rho}+\mu_{-\rho}=4k-4$ then $\mu_{-\rho}$ and $\mu_{\rho}$ have the same parity. Hence $\rvert\mu_{\rho}-\mu_{-\rho}\rvert\in\{0,2\}$. 
		
		We will prove that $\mu_{\rho}=\mu_{-\rho}$.
		Suppose, for contradiction, that $\rvert\mu_{\rho}-\mu_{-\rho}\rvert=2$. By Lemma \ref{integercoefficient}, $\rho$ is an integer. Moreover, $\phi_{S}(x)$ is an integral polynomial. Then $(x-\alpha)(x-\beta)(x-\gamma)$ is an integral polynomial. Hence, $\alpha\beta\gamma$ is an integer. 
		Since $n\geq 8$, we have $k\geq 3$. By \eqref{eq15}, we get $\rho^{2}\in\{4k-1, 4k, 4k+1\}$. From Proposition \ref{det}, ${\rm det}(S)=\pm\alpha\beta\gamma\rho^{4k-4}\equiv2\pmod4$.
		Thus $\rho$ is odd. Hence $\rho^{2}=4k+1$. By \eqref{eq12}, we get 
		\begin{equation}\label{eq13}
			\alpha^{2}+\beta^{2}+\gamma^{2}=6.	
		\end{equation}
		By Cauchy inequality applied to $\begin{pmatrix}
			\alpha\\
			\beta \\
			\gamma
		\end{pmatrix}$
		and $\begin{pmatrix}
			1\\
			1\\
			1
		\end{pmatrix}$
		we get $\rvert\alpha+\beta+\gamma\rvert\leq 3\sqrt{2}$. But from \eqref{trace}, we have $\rvert\alpha+\beta+\gamma\lvert=2\rho=2\sqrt{4k+1}$. This is a contradiction because $k\geq3$. Hence $\mu_{-\rho}=\mu_{\rho}$.
		
		By \ref{trace}, we have 
		\begin{equation}\label{eq14}
			\alpha+\beta+\gamma=0.
		\end{equation}
		By Lemma \ref{integercoefficient}, $\rho^{2}$ is an integer. Moreover, $\phi_{S}(x)$ is an integral polynomial. Then $(x-\alpha)(x-\beta)(x-\gamma)$ is an integral polynomial. Hence, $\alpha\beta\gamma$ is an integer. 
		From Proposition \ref{det}, ${\rm det}(S)=\alpha\beta\gamma\rho^{4k-4}\equiv2\pmod4$.
		Thus $\rho^{2}$ is odd. Then, by \eqref{eq15}, we have $\rho^{2}\in\{4k-1, 4k+1\}$.
		
		We will prove that $\rho^{2}=4k+1$.
		Suppose, for contradiction, that $\rho^{2}=4k-1$. Let $S := [s_{ij}]$ and $B := (4k - 1)I - S^2 = [b_{ij}]$. The matrix $B$ is symmetric and its eigenvalues are $(4k - 1) - \alpha^2$, $(4k - 1) - \beta^2$, $(4k - 1) - \gamma^2$ and $0$. The first three eigenvalues are positive and the multiplicity of $0$ is $4k-4$.
		Thus $B$ is a positive semi-definite matrix with rank equal to 3.
		Moreover, for $i\in\{1, 2,\dots,n\}$ we have
		\[
		b_{ii} = 4k - 1 - (s_{i1}^2 + \ldots + s_{in}^2) = 1.
		\]
		Since $B$ is positive semi-definite, for $i\neq j\in\{1, 2,\dots,n\}$ we have
		\[
		\det \begin{pmatrix}
			1 & b_{ij} \\
			b_{ij} & 1
		\end{pmatrix} \geq 0.
		\]
		Then, $ b_{ij} \in \{0, 1, -1\} $.
		We also have
		\[
		b_{ij} = -\sum_{k=1}^n s_{ik} s_{kj} = -\sum_{k \neq i,k \neq j} s_{ik} s_{kj}.\]
		As the off-diagonal entries of $S$ are $\pm1$, $b_{ij}\equiv 4k - 3 \equiv 1\pmod 2$ and hence $ b_{ij} \in \{1, -1\} $.
		It is easy to check that the determinant of a $3\times 3$ symmetric matrix with $1$'s on the diagonal and $\pm1$ off-diagonal is equal to $0$ or $-4$.
		Moreover, the matrix $B$ is positive semi-definite, then its principal minors of order $3$ must equal to $0$. Then, the rank of $B$ is at most $2$, which leads to a contradiction. Hence $\rho^{2}=4k+1$.
		
		By \eqref{eq12}, we get 
		\begin{equation}\label{eq16}
			\alpha^{2}+\beta^{2}+\gamma^{2}=6.
		\end{equation}
		Moreover, we have $\alpha^{2}\beta^{2}\gamma^{2}\leq(\frac{\alpha^{2}+\beta^{2}+\gamma^{2}}{3})^{3}$. Then $\rvert\alpha\beta\gamma\lvert\leq2$. Since $\alpha\beta\gamma$ is an integer and $\det(S)=\alpha\beta\gamma\rho^{4k-4}\equiv2\pmod{4}$, then $\alpha\beta\gamma\in\{-2,2\}$.
		
		case 1 : $\alpha\beta\gamma=2$, together with the above equations \eqref{eq16} and \eqref{eq14}, we find that $\{\alpha, \beta, \gamma\} = \{[2]^{1}, [-1]^{2}\}$. Hence, $\phi_{S}(x)=(x-2)(x + 1)^{2}(x^{2} - 4k -1)^{2(k -1)}$.
		
		case 2 : $\alpha\beta\gamma=-2$, we get $\{\alpha, \beta, \gamma\} = \{[-2]^{1}, [1]^{2}\}$. Hence, $\phi_{S}(x)=(x+2)(x - 1)^{2}(x^{2} - 4k -1)^{2(k -1)}$.
	\end{enumerate}
	The converse of the four assertions follows from Theorem \ref{index}.
\end{proof}

Analyzing the proof of Theorem \ref{suda}, one can see that the matrices in Theorem \ref{Seidel_thm} are principal sub-matrices of a conference matrix.

\bibliography{symmetric}

\begin{thebibliography}{10}

\bibitem{belevitch1950theorem}
V.~Belevitch.
\newblock Theorem of $2n$-terminal networks with application to conference
  telephony.
\newblock {\em Electrical Communication}, 27:231--244, 1950.

\bibitem{bollobas1981graphs}
B{\'e}la Bollob{\'a}s and Andrew Thomason.
\newblock Graphs which contain all small graphs.
\newblock {\em European Journal of Combinatorics}, 2(1):13--15, 1981.

\bibitem{boussairi2024quasi}
Abderrahim Boussa{\"\i}ri, Brahim Chergui, Zaineb Sarir, and Mohamed Zouagui.
\newblock Quasi-orthogonal extension of skew-symmetric matrices.
\newblock {\em arXiv preprint arXiv:2410.19594}, 2024.

\bibitem{ch2007handbook}
Colbourn Ch and J~Dinitz.
\newblock {\em Handbook of Combinatorial Designs, (Discrete mathematics and its
  applications, ser. ed. K. Rosen)}.
\newblock CRC Press, Boca Raton, FL, 2007.

\bibitem{graham1971constructive}
Ronald~L Graham and Joel~H Spencer.
\newblock A constructive solution to a tournament problem.
\newblock {\em Canadian Mathematical Bulletin}, 14(1):45--48, 1971.

\bibitem{greaves2017symmetric}
G.~Greaves and S.~Suda.
\newblock Symmetric and skew-symmetric $\{0,\pm 1\}$-matrices with large
  determinants.
\newblock {\em Journal of Combinatorial Designs}, 25(11):507--522, 2017.

\bibitem{greaves2016equiangular}
Gary Greaves, Jacobus~H Koolen, Akihiro Munemasa, and Ferenc
  Sz{\"o}ll{\H{o}}si.
\newblock Equiangular lines in euclidean spaces.
\newblock {\em Journal of Combinatorial Theory, Series A}, 138:208--235, 2016.

\bibitem{haantjes1948equilateral}
Johannes Haantjes.
\newblock Equilateral point-sets in elliptic two-and three-dimensional spaces.
\newblock {\em Nieuw Arch. Wiskunde (2)}, 22:355--362, 1948.

\bibitem{koukouvinos1999new}
Christos Koukouvinos and Jennifer Seberry.
\newblock New weighing matrices and orthogonal designs constructed using two
  sequences with zero autocorrelation function--a review.
\newblock {\em Journal of statistical planning and inference}, 81(1):153--182,
  1999.

\bibitem{taussky1971sums}
Olga Taussky.
\newblock sums of squares and hadamard matrices.
\newblock {\em Combinatorics (Proc. Sympos. Pure Math., Vol. XIX, Univ.
  California, Los Angeles, Calif., 1968)}, pages 229--233, 1971.

\bibitem{van1991equilateral}
Jacobus~H van Lint and Johan~J Seidel.
\newblock Equilateral point sets in elliptic geometry.
\newblock In {\em Geometry and Combinatorics}, pages 3--16. Elsevier, 1991.

\end{thebibliography}

\end{document}